\documentclass[12pt,twoside]{amsart} 

\usepackage{amssymb,latexsym,bbm,graphicx,epsfig,epic,eepic,oldgerm,psfrag,comment}
\usepackage{amsmath,amsfonts,amscd,mathrsfs,amsthm}
\usepackage{mathtools}
\usepackage{mathtools}
\usepackage{forest}

\usepackage{tikz,wasysym}  \usetikzlibrary{trees}
\usepackage{marvosym}                  
\usepackage{hyperref} 

\usepackage{xypic} 
\input xy
\xyoption{all}

\theoremstyle{plain}
\newtheorem{theorem}{Theorem}[section]
\newtheorem{lemma}[theorem]{Lemma}                           
\newtheorem{proposition}[theorem]{Proposition}

\newtheorem*{remark*}{Remark}
\newtheorem*{remarks*}{Remarks}

\newtheorem*{example*}{Example}
\newtheorem*{examples*}{Examples}

\newtheorem*{definition*}{Definition}

\newtheorem*{claim*}{Claim}

\numberwithin{figure}{section}
\numberwithin{equation}{section}

\newcommand{\proofend}{\hspace*{\fill} $\square$\\}

\def\1{\:\!}
\def\2{\;\!}

\def\Diffc0{\operatorname{Diff^c_0}}

\def\Sympc0{\operatorname{Symp^c_0}}

\def\cl{{\mathcal L}}

\def\CC{\mathbb{C}}

\def\LL{\mathbb{L}}
\def\NN{\mathbb{N}}

\def\RR{\mathbb{R}}

\def\ZZ{\mathbb{Z}}

\begin{document}

\title{\vspace*{0cm}On the Gromov width of complements of Lagrangian tori}

\date{\today}

\author{Richard Hind}
\address{Department of Mathematics\\
University of Notre Dame\\
255 Hurley\\
Notre Dame, IN 46556, USA.}

\thanks{The author is partially supported by Simons foundation grant number 633715.}

\begin{abstract}
An integral product Lagrangian torus in the standard symplectic $\CC^2$ is defined to be a subset $\{ \pi|z_1|^2 = k, \, \pi|z_2|^2 =l \}$ with $k,l \in \mathbb{N}$.
Let $\cl$ be the union of all integral product Lagrangian tori. We compute the Gromov width of complements $B(R) \setminus \cl$ for some small $R$, where $B(R)$ denotes the round ball of capacity $R$.
\end{abstract}

\maketitle

\section{Introduction}

The Gromov width is a fundamental quantitative invariant of symplectic manifolds, and a basic problem is to detect changes in the width when we restrict to the complement of Lagrangian submanifolds. In this form the question goes back to Biran \cite{biran}, and if the Gromov width of the complement is smaller than that of the whole symplectic manifold we say the Lagrangian is a barrier.

In this paper we study the complement of disjoint unions of Lagrangian tori.  This is of interest to dynamical systems. Wandering domains are open sets which remain disjoint from themselves under iterations, say of a time 1 flow. The article \cite{lms} of Lazzarini, Marco and Sauzin studied wandering domains in Hamiltonian perturbations of integrable systems. Integrable systems themselves have no wandering domains, since the Hamiltonian flow acts linearly on the invariant tori. More generally, after a perturbation, any wandering domain must be disjoint from all KAM tori. Nevertheless Arnold diffusion shows wandering domains may exist after perturbation, and  
quantitative estimates on the size of wandering domains were obtained in \cite{lms}. In the same paper the authors raise the question of determining the Gromov width of the complement of the invariant tori; upper bounds on the Gromov width give a priori estimates on the symplectic size of wandering sets. See also the article of Schlenk, \cite{scha} section 4.4, for a discussion of this.

Here we study the simplest example, when the ambient symplectic manifold is $\RR^4$. We identify $\RR^4$ and $\CC^2$ and use coordinates $(z,w)$ so the standard symplectic form is $\omega = \frac{i}{2}(dz \wedge d\overline{z} + dw \wedge d\overline{w})$. Then we can define a ball of capacity $a$ by $$B(a) = \{ \pi( |z|^2 + |w|^2) < a\}.$$ For $U \subset \CC^2$ the Gromov width $$c_G(U) = \sup \{a \, | \, B(a) \hookrightarrow U \}$$ where throughout $V \hookrightarrow U$ denotes the existence of a symplectic embedding of $V$ into $U$.

We will study the complement of Lagrangian tori $$L(k,l) = \{\pi |z|^2 = k, \, \pi |w|^2 =l\}$$ for $k,l \in \NN$. We call these the integral Lagrangian tori. Let $\cl = \cup L(k,l)$. By inclusion we have embeddings $B(a) \hookrightarrow \CC^2 \setminus \cl$ for all $a<2$, however before the current work it seems no embeddings of larger balls were known, and so one could ask whether $c_G(\CC^2 \setminus \cl) =2$.

Our approach is to restrict the ambient space to a ball $B(R)$ and begin an analysis of how $c_G(B(R) \setminus \cl)$ varies with $R$. We note that this function is clearly nondecreasing, but is not obviously continuous.

\begin{theorem}\label{main}
\begin{equation} \label{mainform}
c_G(B(R) \setminus \cl) =
\left\{
\begin{array}{lcl}
R    & R \le 2;    \\ [0.2em]
2    & 2< R \le 3;    \\ [0.2em]
R-1    & 3< R \le 4;    \\ [0.2em]
3    & 4< R \le 5;    \\ [0.2em]
\ge R-2    & 5< R \le 6;    \\ [0.2em]
4  & R=6 .
\end{array}
\right.
\end{equation}
\end{theorem}

Given this, one naturally asks whether it is possible to embed balls of capacity greater than $4$, or conversely whether $c_G(B(R) \setminus \cl)$ is constant when $R>6$ and $c_G(\CC^2 \setminus \cl) =4$. We cannot answer this, and it also remains open as to whether $c_G(\CC^2 \setminus \cl) = \infty$.

The proof of Theorem \ref{main} consists of establishing upper bounds (using holomorphic curve methods) and lower bounds (from explicit constructions) for each of the statements. Only the first line is clear, together with the lower bound for line 2. For lower bounds then, given $c_G(B(R) \setminus \cl)$ is nondecreasing, it suffices to deal with lines 3 and 5.

The most delicate upper bound is for line 4. In the cases of lines 2, 3 and 6 we actually obtain the same upper bound for balls avoiding a single monotone torus. 
Precisely, when $2<R \le 4$ we have
$$c_G(B(R) \setminus \cl) = c_G(B(R) \setminus L(1,1))$$
and setting $R=6$ we have
$$c_G(B(6) \setminus \cl) = c_G(B(R) \setminus L(2,2)).$$

We note that for $4 < R \le 5$ the ball $B(R)$ contains $6$ integral Lagrangian tori. It is easy to see that we can embed a closed ball of capacity $3$ avoiding any single one of these tori, but it turns out to be impossible to find an embedding of the closed ball avoiding for example $L(1,2) \cup L(2,2)$. By a result of McDuff \cite{Mc91} all ball embeddings $B(a) \hookrightarrow B(R)$ are Hamiltonian isotopic. Therefore the upper bound can be cast as a Lagrangian packing obstruction, or a necessary intersection in $B(5) \setminus B(3)$ between Lagrangians with different area classes. 

\medskip

{\bf Acknowledgements.} The author thanks Karim Boustany for many enjoyable conversations on these topics.

\section{Constructions}

Here we establish the lower bounds in Theorem \ref{main}. The nontrivial lower bounds follow from those in the third and fifth lines, and specifically we need to construct symplectic embeddings
$$B(a) \hookrightarrow B(R) \setminus \cl$$
when $3<R\le 4$ and $a< R-1$, and when $5<R \le 6$ and $a<R-2$.

\begin{proposition}\label{con1} Let $3<R \le 4$ and $a<R-1$. Then there exists a symplectic embedding
 $$B(a) \hookrightarrow B(R) \setminus (L(1,1) \cup L(1,2) \cup L(2,1)).$$
\end{proposition}

\begin{proof} We start by considering the inclusion $B(a) \subset B(R)$. Since $a< R-1 \le 3$ we see that $L(1,2) \cup L(2,1)$ is disjoint from $B(a)$. Our goal is to find a Hamiltonian diffeomorphism of $B(R)$, with support away from $L(1,2) \cup L(2,1)$, displacing $L(1,1)$ from $B(a)$.

To do this we will construct two Hamiltonian diffeomorphisms. The first, $\phi$, satisfies $\phi(L(1,1)) \cap B(a) = \emptyset$ and the second $\psi$ has support in the complement of $B(a)$ and satisfies $\psi(L(1,2) \cup L(2,1)) \cap \phi^t(L(1,1)) = \emptyset$ for all $t$. Here $\{ \phi^t\}$ is the Hamiltonian flow with $\phi^0 = \mathrm{Id}$ and $\phi^1 = \phi$. Then the Hamiltonian flow $\zeta^t = \psi^{-1} \circ \phi^t \circ \psi$ is as required.
In other words, $\zeta^1(L(1,1)) \subset B(R) \setminus B(a)$ and $\zeta^t(L(1,1))$ is disjoint from $L(1,2) \cup L(2,1)$ for all $t$. Therefore the isotopy $\zeta^t(L(1,1))$ can be generated by a Hamiltonian diffeomorphism with support away from $L(1,2) \cup L(2,1)$.

\vspace{0.1in}

{\bf 1. Construction of $\phi$.}

\vspace{0.1in}

This is generated by a time independent Hamiltonian of the form $H(z,w) = H(z)$. The function $H(z)$ has support in the round disk $D(R-1)$ of area $R-1$, and is such that the corresponding Hamiltonian diffeomorphism $\phi_H$ displaces $D(1)$ from $D(a-1)$. Our bounds imply that $(a-1) + 1 = a < R-1$ and so such a Hamiltonian exists. We also assume that the trace of $\phi_H^t(D(1))$ for $0 \le t \le 1$ intersects $\partial D(2)$ only near $z=\sqrt{ 2 / \pi}$ on the real axis, see Figure \ref{f1}.

Moving to $\CC^2$, the flow $\phi^t$ of $H(z,w)$ is such that $\phi^t(L(1,1))$ intersects a neighborhood of $\partial B(3)$ only near the cylinder $T =\{ \mathrm{Im}(z)=0  \} \times \partial D(1)$ and $\phi^1$ displaces $L(1,1)$ from $B(a)$.

\begin{figure}[h!]
\vspace{.2cm}
\centering
\includegraphics[width=120mm,trim={0cm 0cm 0cm 0cm},clip]{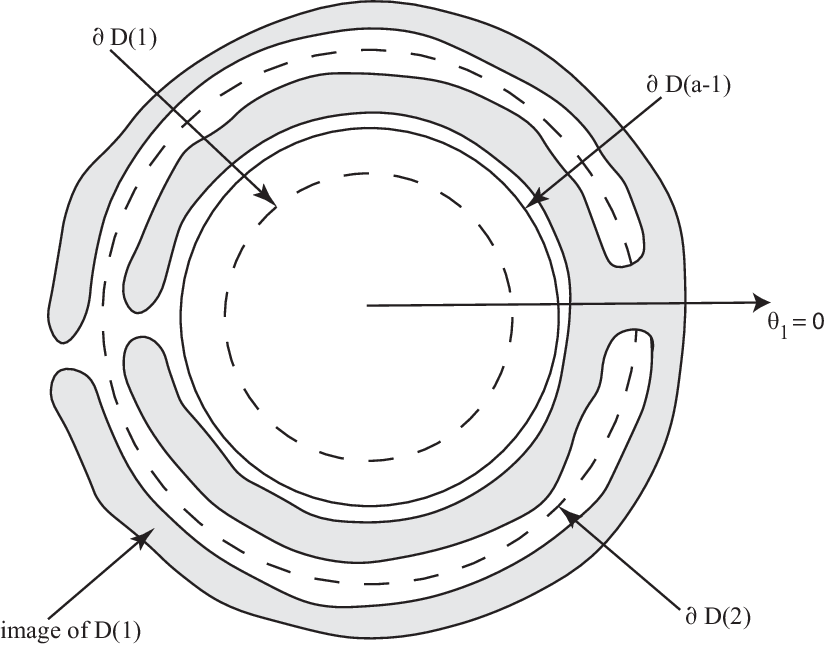}
\caption{The image of $D(1)$ under $\phi_H$.}
\label{f1}
\end{figure}

\vspace{0.1in}

{\bf 2. Construction of $\psi$.} 

\vspace{0.1in}

We construct a Hamiltonian diffeomorphism with support near the sphere $\partial B(3)$ displacing $L(1,2) \cup L(2,1)$ from a neighborhood of the cylinder $T$. 
This is equivalent to constructing a Hamiltonian isotopy of $L(1,2) \cup L(2,1)$, where the Lagrangian tori remain close to $\partial B(3)$.
It is convenient to use symplectic polar coordinates $(R_1, \theta_1, R_2, \theta_2)$ where $R_1 = \pi|z|^2$ and $R_2 = \pi|w|^2$. Also $\theta_1, \theta_2 \in \RR / \ZZ$. Then $T= \{\theta_1 =0, R_2=1\}$.

Consider the symplectic reduction $p : \partial B(3) \to S$, that is, the quotient by the characteristic circles $\Gamma_{t, c} = \{R_1 - R_2 = t, \theta_1 - \theta_2 = c\}$. Then $S$ is a $2$-sphere of area $3$. We can use polar coordinates $(t,c) \in [-3,3] \times S^1$ on $S$, so the reduced symplectic form is $\frac 12 dt \wedge dc$. The product Lagrangians $L(a,b)$ project to circles $\{t=a-b\}$ and Hamiltonian isotopies of these circles in $S$ lift to Hamiltonian isotopies of the tori $L(a,b)$ in $\CC^2$.

With this notation, we have $$p(L(1,2)) = \{t=-1\}, \, \,  p(L(2,1)) = \{t=1\}, \, \, p(T \cap \partial B(3)) = \{t=1\}.$$
The projections of $L(1,2)$ and $L(2,1)$ are shown in Figure \ref{f2} together with their images $\Gamma_1$ and $\Gamma_2$ under a Hamiltonian flow. 
(Rather, the figure shows the parts of these circles in one hemisphere.) 
Hence $\Gamma_1$ and $\Gamma_2$ still bound disks of area $1$, the circle $\Gamma_1 \subset \{t<1\}$ while $\Gamma_2$ intersects $\{t \ge 1\}$ only in a small strip, say near $\{c=0\}$. Lifting the isotopy, the image $\LL_1 = p^{-1}(\Gamma_1)$ of $L(1,2)$ is disjoint from $T$ and the image $\LL_2 = p^{-1}(\Gamma_2)$ of $L(2,1)$ intersects $\{R_1 \ge 2\}$ only near the characteristic circles $\Gamma_{t,0}$ with $t \ge 1$.

\begin{figure}[h!]
\vspace{.2cm}
\centering
\includegraphics[width=120mm,trim={0cm 0cm 0cm 0cm},clip]{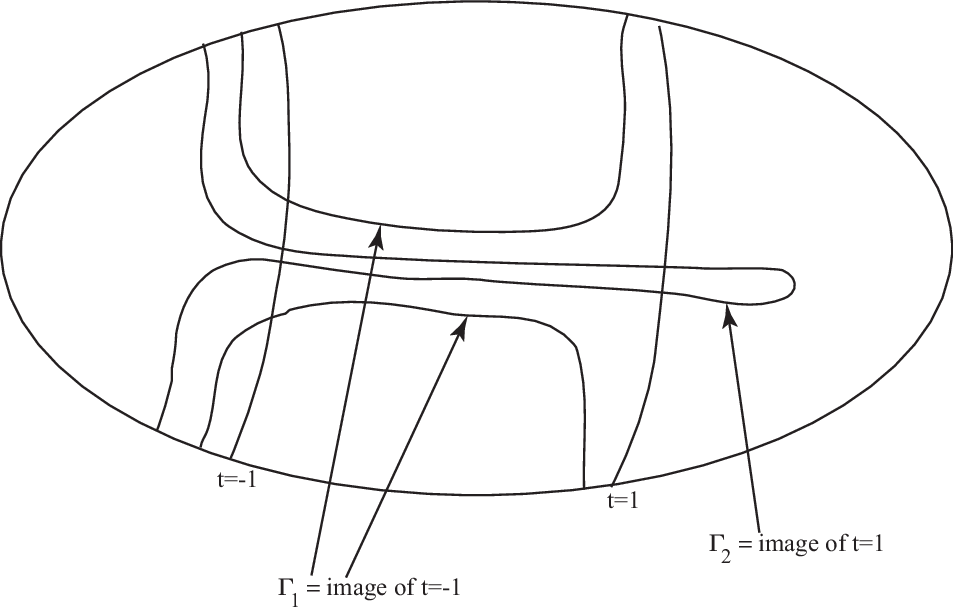}
\caption{Hamiltonian isotopy in the symplectic reduction of $\partial B(3)$.}
\label{f2}
\end{figure}

Finally we apply to $\LL_2$ the diffeomorphism generated by a Hamiltonian $G(\theta_2)$, where $G' = - \epsilon$ near $\theta_2=0$ and $|G'| < \delta$ elsewhere, with $\delta$ small. The corresponding Hamiltonian vector field is $-G'(\theta_2)\partial_{R_2}$, so points of $\LL_2$ which are moved by the flow are pushed away from $\partial B(3)$, and in particular remain disjoint from $\LL_1$. It remains to check that $\LL_2$ is displaced from the trace $\{ \phi^t(L(1,1)) \}$. For $\epsilon$, $\delta$ small any intersections occur near $\partial B(3)$, so we need to check that $\LL_2$ is displaced from $T$.

First, as $\theta_1$ is invariant under the flow, it suffices to consider points of $\LL_2$ with $\theta_1$ close to $0$. Next, as $G' < \delta$, the only points mapping to $T$ must initially have $R_2 < 1 + \delta$, or equivalently $R_1 > 2 - \delta$. Hence we only need to consider points of $\LL_2$ in the preimage of the thin strip near $t=1$. But here the coordinate $c = \theta_1 - \theta_2$ is close to $0$, so we are considering points with both $\theta_1$ and $\theta_2$ small. This means $G' = -\epsilon$ and such points move a distance $\epsilon$ in the positive $R_2$ direction. Initially we may assume $R_1<2+\epsilon$, or equivalently $R_2>1-\epsilon$,
so the image of such points lies in $\{R_2>1\}$, in particular disjoint from $T$ as required.
\proofend
\end{proof}

\begin{proposition}\label{con2} Let $5<R \le 6$ and $a<R-2$. Then there exists a symplectic embedding
 $$B(a) \hookrightarrow B(R) \setminus \cl.$$
\end{proposition}

\begin{proof} We argue similarly to Proposition \ref{con1}. To simplify the descriptions we assume $a>3$.

First we embed $B(a)$ by inclusion. Since $a<4$ the product Lagrangians lying inside $B(a)$ are $L(1,1)$, $L(1,2)$ and $L(2,1)$. Our goal is to find a Hamiltonian diffeomorphism of $B(R)$ which displaces $L(1,1) \cup L(1,2) \cup L(2,1)$ from $B(a)$, and which has support in the complement of the $L(k,l)$ with $k+l \in \{4, 5\}$. To do this, we first describe a Hamiltonian flow $\phi^t$ displacing $L(1,1) \cup L(1,2) \cup L(2,1)$ from $B(a)$, then a second Hamiltonian diffeomorphism $\psi$ with compact support in $B(R) \setminus B(a)$ and such that $\psi(L(k,l))$ is disjoint from the trace $\phi^t(L(1,1) \cup L(1,2) \cup L(2,1))$ for all $0 \le t \le 1$ and $k+l \in \{4, 5\}$. Then the Hamiltonian flow $\psi^{-1} \circ \phi^t \circ \psi$ will be as required.

\vspace{0.1in}

{\bf 1. Construction of $\phi$.}

\vspace{0.1in}

Let $H(z)$ be a Hamiltonian function with compact support in $D(R-1)$ such that the corresponding diffeomorphism displaces $D(2)$ from $D(a-1)$. For the existence of such $H$ we are using the assumption that $2 + (a-1) < R-1$. We may further assume that the trace $\phi_H^t(D(1))$ for $0 \le t \le 1$ intersects $\partial D(3) \cup \partial D(4)$ only near $z=\sqrt{ 3 / \pi}$ and $z=\sqrt{ 4 / \pi}$  on the real axis, see Figure \ref{f3}.

\begin{figure}[h!]
\vspace{.2cm}
\centering
\includegraphics[width=120mm,trim={0cm 0cm 0cm 0cm},clip]{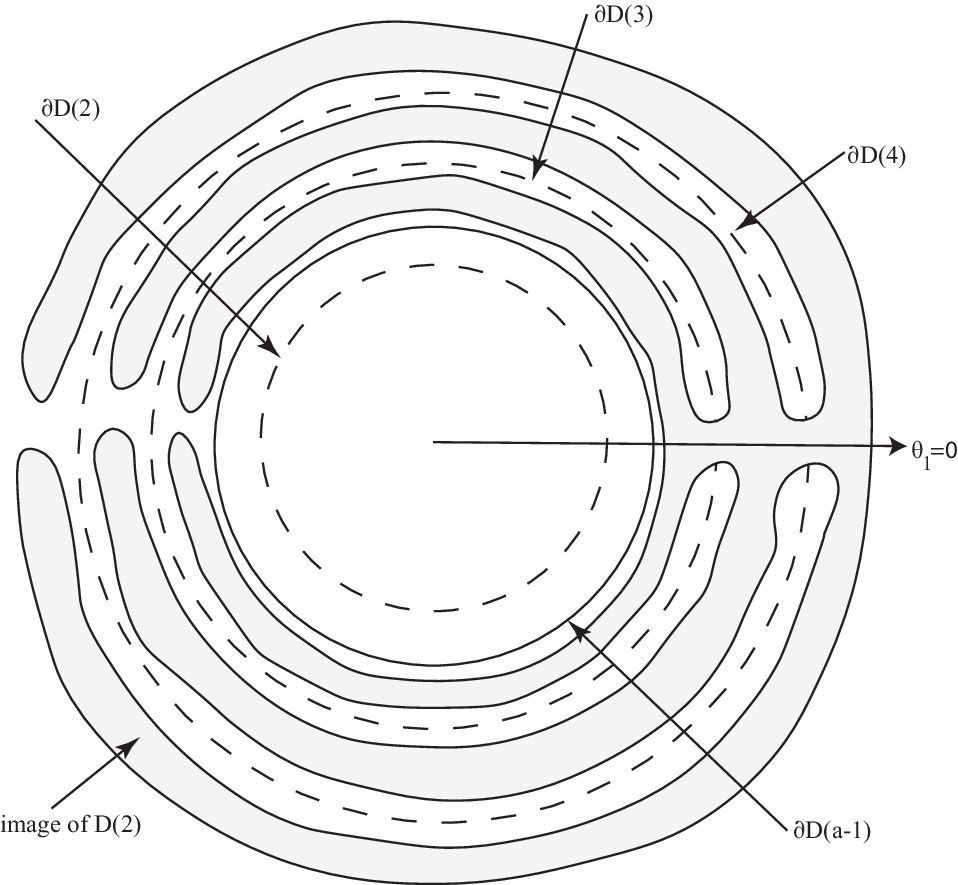}
\caption{The image of $D(1)$ under $\phi_H$.}
\label{f3}
\end{figure}

Next we move to $\CC^2$ and let
$\phi_1^t$ be the flow generated by the Hamiltonian function $H(z,w) = H(z)$. This satisfies $$\phi_1^1(L(1,1) \cup L(2,1)) \subset \{R_1 > a-1, \, R_2=1\}$$ and $\phi_1^t(L(1,1) \cup L(2,1)) \cap (\partial B(4) \cup \partial B(5))$ lies in a neighborhood of $\{ \theta_1=0\}$, where we are using symplectic polar coordinates as above.

Similarly, let $\phi_2^t$ be the flow generated by a Hamiltonian function $H(z,w) = H(w)$ such that $$\phi_2^1(L(1,2)) \subset \{R_2 > a-1, \, R_1=1\}$$ and $\phi_2^t(L(1,2)) \cap (\partial B(4) \cup \partial B(5))$ lies in a neighborhood of $\{ \theta_2=0\}$. Since the trace $\phi_2^t(L(1,2)) \subset \{R_1=1\}$ and $\phi_1^1(L(1,1) \cup L(2,1)) \subset \{R_1 > a-1 \}$, we can follow $\phi_1^t$ then $\phi_2^t$ to find a single (time dependent) Hamiltonian flow $\phi^t$ where the trace of $\phi^t(L(1,1) \cup L(2,1))$ coincides with $\phi^t_1(L(1,1) \cup L(2,1))$ and the trace of $\phi^t(L(1,2))$ coincides with $\phi^t_2(L(1,2))$.

\vspace{0.1in}

\newpage

{\bf 2. Construction of $\psi$.}

\vspace{0.1in}

Define cylinders $$T_1 =   \{R_1=1, \theta_2=0\}, \, \, T_2 =\{R_2=1, \theta_1=0\}.$$ Then $\{ \phi^t_2(L(1,2)) \}$ intersects $\partial B(4)$ and $\partial B(5)$ only near $T_1$, and $\{ \phi^t_1(L(1,1) \cup L(2,1)) \}$ intersects $\partial B(4)$ and $\partial B(5)$ only near $T_2$. 
We aim to find Hamiltonian diffeomorphisms $\psi_1$ and $\psi_2$ supported near $\partial B(4)$ and $\partial B(5)$ respectively, so that $$\psi_1(L(1,3) \cup L(2,2) \cup L(3,1)) \cap (T_1 \cup T_2) = \emptyset$$ and $$\psi_2(L(1,4) \cup L(2,3) \cup L(3,2) \cup L(4,1)) \cap (T_1 \cup T_2) = \emptyset.$$

\vspace{0.1in}

{\bf 2a. The map $\psi_1$.}

\vspace{0.1in}

As above we consider the symplectic reduction $p : \partial B(4) \to S$, where $S$ is now a sphere of area $4$. Now this means we quotient by the characteristic circles $\Gamma_{t, c} = \{R_1 - R_2 = t, \theta_1 - \theta_2 = c\}$ where $t \in [-4,4]$ and $c \in \RR / \ZZ$. We can use polar coordinates $(t,c)$ on $S$, so the reduced symplectic form is $\frac 12 dt \wedge dc$. The product Lagrangians $L(a,b)$ project to circles $\{t=a-b\}$ and Hamiltonian isotopies of these circles lift to Hamiltonian isotopies of the tori $L(a,b)$. If $k+l=4$, the projection of $L(k,l)$ is the circle $\{ t = k-l \}$.

We apply a Hamiltonian diffeomorphism mapping $\{t=-2\}$ and $\{t = 2\}$ to circles $\Gamma_1$ and $\Gamma_2$ as shown in Figure \ref{f4}. We see that $\Gamma_1$ and $\Gamma_2$ bound disks $D_1$ and $D_2$ of area $1$, and, topologically, $A= \{-2 < t < 2\} \setminus (D_1 \cup D_2)$ is an annulus which we may assume contains the image of the circle $\{t=0\}$ under our Hamiltonian diffeomorphism. Based on Figure \ref{f4} we may also assume $\Gamma_1 \subset \{-\frac14 < c < \frac 14\}$ and intersects $\{t< -2\}$ only near $c=0$. Similarly $\Gamma_2 \subset \{\frac14 < c < \frac 34\}$ and intersects $\{t>2\}$ only near $c=\frac 12$.

\begin{figure}[h!]
\vspace{.2cm}
\centering
\includegraphics[width=120mm,trim={0cm 0cm 0cm 0cm},clip]{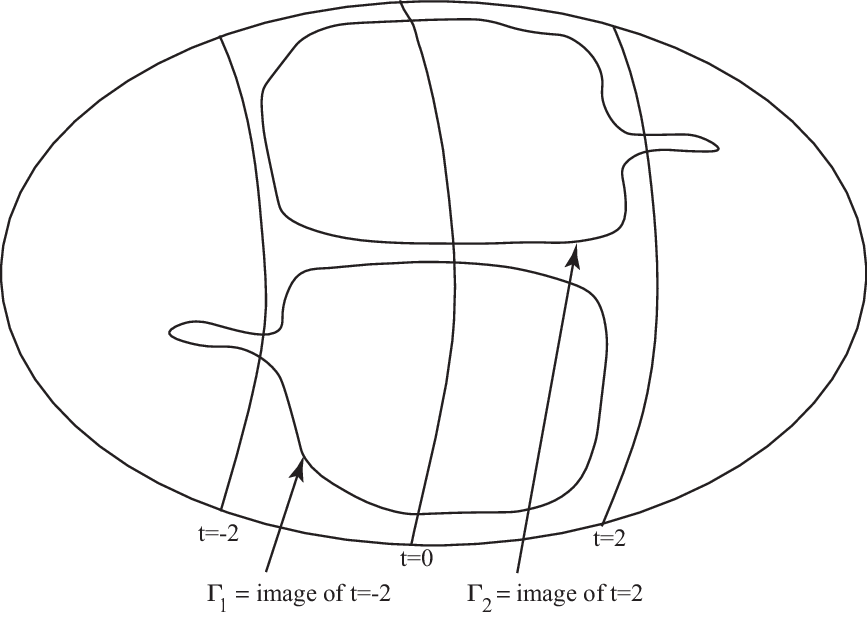}
\caption{Hamiltonian isotopy in the symplectic reduction of $\partial B(4)$.}
\label{f4}
\end{figure}

Lifting the Hamiltonian diffeomorphism, of the images of the three integral Lagrangians $L(k,l)$ with $k+l=4$, only $\LL_1 = p^{-1}(\Gamma_1)$, the image of $L(1,3)$, intersects $T_1$ and only $\LL_2 = p^{-1}(\Gamma_2)$, the image of $L(1,3)$, intersects $T_2$. The image of $L(2,2)$ projects to the image of the circle $\{ t=0 \}$ inside the annulus $A$.

To displace $\LL_1$ from $T_1$ we apply a Hamiltonian diffeomorphism generated by a function $G_1(\theta_1)$, so the Hamiltonian vector field is $-G_1'(\theta_1)\partial_{R_1}$. We may assume $G_1'(\theta_1) = - \epsilon$ when $\theta_1$ is close to $0$ and $G_1'(\theta_1) \le 0$ when $-\frac14 < \theta_1 < \frac 14$.
First note that as $\LL_1 \subset \{R_2>1\}$ and the flow is parallel to the $R_1$ direction, $\LL_1$ remains disjoint from $T_2$.

As $\theta_2$ is an invariant of the motion, to check displacement from $T_1$ we need only consider points
$p \in \LL_1$ with $\theta_2(p)=0$. If $p$ lies in the region $\{ R_1 \le 1 \}$ then it lies on the preimage of the thin strip. Then as $c$ is close to $0$ this means $\theta_1(p)$ is also close to $0$ and the Hamiltonian flow moves $p$ a distance $\epsilon$ in the positive $R_1$ direction, displacing $p$ from $T_1$. In general $\LL_1 \subset  \{-\frac14 < c < \frac 14\}$ and so $\theta_2(p)=0$ implies $\theta_1(p) \in [-\frac14, \frac14]$, and hence the $R_1$ coordinate is nondecreasing under the flow, that is, points $p \in \{R_1>1\}$ remain disjoint from $T_1$. 

To displace $\LL_2$ from $T_2$ we argue in the same way. Apply a Hamiltonian diffeomorphism generated by a function $G_2(\theta_2)$, where $G_2'(\theta_2) = - \epsilon$ when $\theta_2$ is close to $\frac 12$ and $G_1'(\theta_2) \le 0$ when $\frac14 < \theta_2 < \frac 34$. 
As the Hamiltonian flow is parallel to the $R_2$ direction, the flow of $\LL_2$ remains disjoint from $T_1$. For intersections with $T_2$ we consider
points $q \in \LL_2$ with $\theta_1(q)=0$. If $q$ lies in the region $R_2 \le 1$ then it lies on the preimage of the thin strip. As $c$ is close to $\frac 12$ here, this means $\theta_2(q)$ is close to $\frac 12$, and the Hamiltonian flow moves $q$ a distance $\epsilon$ in the positive $R_2$ direction, displacing $q$ from $T_2$. In general we have $q \in \{\frac14 < c < \frac 34\}$ and so $\theta_1(q)=0$ implies $\theta_2(q) \in [\frac14, \frac34]$. Hence the $R_2$ coordinate is nondecreasing under the flow, and points $q \in \{ R_2 \ge 1\}$ remain disjoint from $T_2$.

Finally we observe that these flows applied to $\LL_1$ and $\LL_2$ do not create new intersections between the Lagrangians. Indeed, as $\theta_1$ and $\theta_2$ are invariants of the flow, $\LL_1$ and $\LL_2$ remain disjoint. But any points which are moved by the flow are moved either only in the $R_1$ direction (for $\LL_1$) or only in the $R_2$ direction (for $\LL_2$), and thus remain disjoint from the image of $L(2,2)$ (which lies in $\partial B(4)$).

\vspace{0.1in}

{\bf 2b. The map $\psi_2$.}

\vspace{0.1in}

Here we consider the symplectic reduction $p : \partial B(5) \to S$, so $S$ is a sphere of area $5$. As usual this means we quotient by the characteristic circles $\Gamma_{t, c} = \{R_1 - R_2 = t, \theta_1 - \theta_2 = c\}$ where $t \in [-5,5]$ and $c \in \RR / \ZZ$. 
The product Lagrangians $L(a,b)$ project to circles $\{t=a-b\}$. 
Hence, if $k+l=5$, the projection of $L(k,l)$ is the circle $\{ t = k-l \}$.

We apply a Hamiltonian diffeomorphism mapping $\{t=-3\}$ and $\{t = 3\}$ to circles $\Gamma_1$ and $\Gamma_2$ as shown in Figure \ref{f5}. We see that $\Gamma_1$ and $\Gamma_2$ bound disks $D_1$ and $D_2$ of area $1$, and, topologically, $A= \{-3 < t < 3\} \setminus (D_1 \cup D_2)$ is an annulus of area slightly greater than $1$. (The excess is the area of the thin strips of $D_1$ and $D_2$ which intersect the regions $\{ t<-3\}$ and $\{t >3\}$.) Hence we may assume $A$ contains the image of the annulus $\{-1<t<1\}$ under our Hamiltonian diffeomorphism.

\begin{figure}[h!]
\vspace{.2cm}
\centering
\includegraphics[width=120mm,trim={0cm 0cm 0cm 0cm},clip]{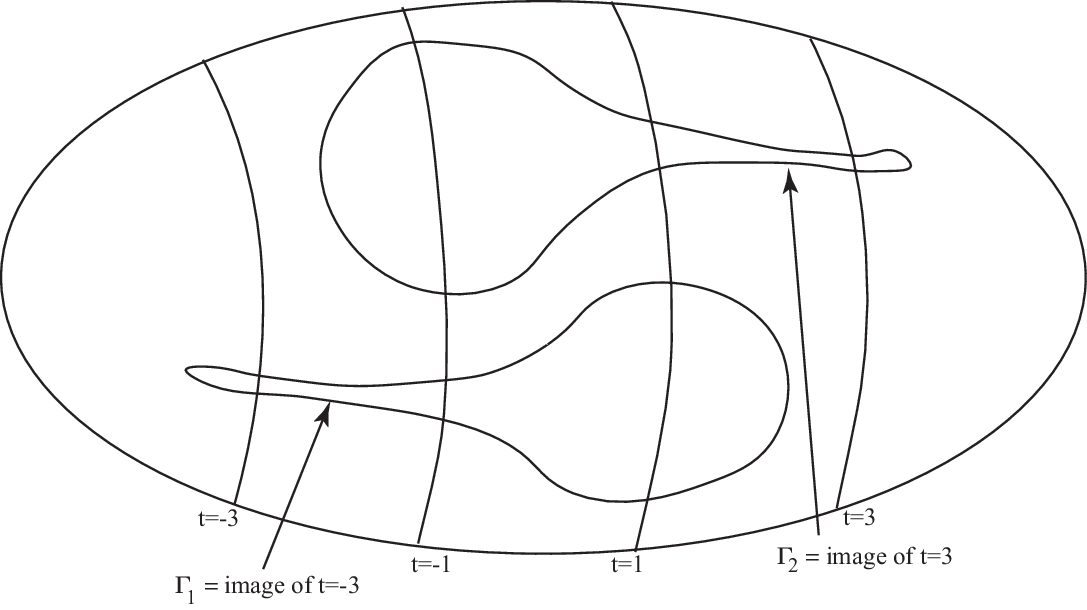}
\caption{Hamiltonian isotopy in the symplectic reduction $\partial B(5)$.}
\label{f5}
\end{figure}

Lifting the isotopy, let $\LL_{k,l}$ be the image of $L(k,l)$ with $k+l=5$. Then only $\LL_{1,4}$ intersects $T_1$ and only $\LL_{4,1}$ intersects $T_2$. The argument now follows as before. We can move $\LL_{1,4}$ slightly in the $R_1$ direction to displace it from $T_1$ and $\LL_{4,1}$ slightly in the $R_2$ direction to displace it from $T_2$. As points are displaced from $\partial B(5)$ this leaves $\LL_{1,4}$ and $\LL_{4,1}$ disjoint from $\LL_{2,3} \cup \LL_{3,2}$, and $\LL_{1,4}$ and $\LL_{4,1}$ remain disjoint since they are only slightly perturbed. 
\proofend
\end{proof}


\section{Obstructions}

\subsection{Balls avoiding a single Lagrangian}

We start with a general bound which in particular implies the upper bounds for lines 1, 2, 3 and 6 of Theorem \ref{main}.

\begin{proposition}\label{ob1} $c_G(B(R) \setminus \cl) \le \min(2\lceil \frac{R}{3} \rceil, \, R - \lfloor \frac{R}{3} \rfloor)$.
\end{proposition}

Let $n$ be an integer with $n < \frac{R}{2}$. Then $L(n,n) \subset B(R)$. We will show the following, then Proposition \ref{ob1} follows by applying the lemma to $n = \lfloor \frac{R}{3} \rfloor$ and $n= \lceil \frac{R}{3} \rceil$.

\begin{lemma} Suppose we have a symplectic embedding $B(a) \hookrightarrow B(R) \setminus L(n,n)$. Then $a \le \max(2n, R-n)$.
\end{lemma}

\begin{proof} We will establish a strict inequality for embeddings which extend to a closed ball $\overline{B(a)}$. This implies the weak inequality for open balls by a restriction.

We compactify the ball $B(R)$ by including it as the affine part of a copy of $\CC P^2$ with lines of area $R$, and then blow up the embedded copy of $\overline{B(a)}$. The resulting symplectic manifold $(X, \omega)$ has $H_2(X, \ZZ)$ generated by the class $G$ of a line and the exceptional divisor $E$. Note that $\omega(G) = R$ and $\omega(E)=a$. Then for any compatible almost complex structure we have a foliation of $X$ by holomorphic spheres in the class $G-E$, hence of area $R - a$. We may assume our almost complex structures make both the line at infinity in $\CC P^2$ and the exceptional divisor complex.

We perform a neck stretching along $L = L(n,n)$. Taking a limit of holomorphic spheres intersecting a fixed point on $L$, the result is a holomorphic building which generically contains two disks in $X \setminus L$ of Maslov class $2$, asymptotic to the same geodesic with opposite orientation. The disks fit together to give a representative of the class $G - E$, and since $(G - E) \cdot G=1$, by positivity of intersection exactly one of these disks intersects the line at infinity. Let $D$ be the other disk, which therefore sits in a single blow up of the ball.

There are two cases. If $D \cap E = \emptyset$ then the Maslov $2$ condition implies it has area $n$. As it appears as part of a limit of spheres of area $R - a$ we see that $a < R - n$. Alternatively $D \cap E = 1$ and the Maslov $2$ condition implies $D$ has area $2n - a$. Hence we have $a < 2n$ and the result follows.
\proofend
\end{proof}

\subsection{Balls avoiding two Lagrangians}

Here we prove the following, which gives the upper bound for line 4 of Theorem \ref{main}, thus completing its proof.

\begin{proposition}\label{int} A symplectic embedding $\overline{B(3)} \hookrightarrow B(5)$ must intersect $L(1,2) \cup L(2,2)$.
\end{proposition}

The proof of this occupies the remainder of the paper. It was shown by Boustany in \cite{karim} that an integral Lagrangian torus in the monotone blow up of $\CC P^2$ must intersect the Clifford torus. In our language, this in particular implies all integral Lagrangian tori in $B(3) \setminus B(1)$ must intersect $L(1,1)$. We can think of Proposition \ref{int} also as a Lagrangian intersection result. Namely, both $L(1,2)$ and $L(2,2)$ can separately be displaced from the ball $B(3)$ inside $B(5)$. However, if we denote their images by $L_1, L_2 \subset B(5) \setminus B(3)$ then we must have $L_1 \cap L_2 \neq \emptyset$. This result does not seem to follow formally from Boustany's, but our proof follows closely.

\subsubsection{Outline.}

We argue by contradiction and suppose there exists an embedded ball $\overline{B(3)} \subset B(5) \setminus (L(1,2) \cup L(2,2))$.

We will blow up the embedded ball $B(3)$ and compactify $B(5)$ to obtain a one point blow up $X$ of $\CC P^2$ with lines of area $5$ and exceptional divisor $E$ of area $3$. Given an almost complex structure $J$, this can be represented as an $S^2$ bundle over the line at infinity $\Sigma$, where the fibers are $J$ holomorphic spheres intersecting each of  $\Sigma$ and $E$ exactly once.

The first step is to examine behavior of this foliation under neck stretching along $L(1,2) \cup L(2,2)$. A $1$ dimensional family of spheres degenerate into `broken leaves', which consist of a pair of Maslov $2$ planes asymptotic to the same geodesic, with opposite orientations, on one of the tori. We are able to determine the homology classes of these broken planes.

Next we study sections of our $S^2$ bundle with very high degree, and how these sections behave under neck stretching. The main result here can be taken from \cite{karim}.

The limiting sections can be deformed to produce smooth sections $F$ and $G$, which after a blowing up and blowing down procedure transform to disjoint sections in a new copy of $\CC P^2  \# \overline{ \CC P^2}$. The complement of the union of these two sections plus two fibers contains our two Lagrangian tori, and can be embedded in $T^* T^2$. We will see that the Lagrangians are relatively exact in the cotangent bundle, which gives a contradiction since they remain disjoint.

\subsubsection{Disks with boundary on $L(1,2) \cup L(2,2)$.}

We start by recording facts about Maslov classes of disks with boundary on $L(1,2)$ or $L(2,2)$. We use the natural bases of $H_1(L(1,2)) \cong \ZZ^2$ and  $H_1(L(2,2)) \cong \ZZ^2$ coming from their product structures.

\begin{lemma} \label{mas} Let $J$ be an almost complex structure on $\CC^2$ tamed by the standard symplectic form.
\leavevmode
\begin{enumerate}
\item Maslov $2$ disks in $\CC^2$ with boundary on $L(1,2)$ and area $1$ have boundary in the class $(1,0)$;\\
\item disks in $\CC^2$ with Maslov class $2n$ and boundary on $L(1,2)$ in a class $(k,1)$ have area $n+1$;\\
\item disks in $\CC^2$ with Maslov class $2n$ and boundary on $L(1,2)$ in a class $(k,-1)$ have area $n-1$;\\
\item  there are no Maslov $4$ $J$-holomorphic disks in $\CC^2$ with area $4$ and embedded boundary on $L(1,2)$;\\
\item  disks in $\CC^2$ with Maslov class $2n$ and boundary on $L(2,2)$ in a class $(k, k \pm 1)$ have area $2n$;\\
\item  $J$-holomorphic Maslov $4$ disks in $\CC^2$ with embedded boundary on $L(2,2)$ have area $4$ and boundary in the class $(1,1)$.\\
\end{enumerate}
\end{lemma}

\begin{proof} We focus on statement (6), the remaining statements are easily checked.

Suppose then we have a $J$ holomorphic Maslov $4$ plane in $\CC^2$ with boundary on $L(2,2)$ in the class $(k,l)$. The Maslov $4$ condition implies $k+l=2$. Then as the boundary is embedded we must have $(k,l)=(1,1)$ if both $k,l \ge 0$. We argue by contradiction assuming, say, that $k<0$.

We find a family of almost-complex structures $J_t$ interpolating between $J_0 = J$ and the standard product structure $J_1$. If the Maslov $4$ planes persist to give a $J_1$ holomorphic plane then as $k<0$ the plane must intersect $\{z_1=0\}$ negatively, a contradiction as the axis is complex with respect to $J_1$.

Alternatively, the plane must degenerate, and as $L(2,2)$ admits no holomorphic planes of Maslov class $0$ (since they would have area $0$) the Maslov $4$ plane bubbles into either a pair of Maslov $2$ planes, or perhaps a Maslov $2$ plane and a cylinder of Fredholm index $2$.

In the first case,
one of the planes must be asymptotic to an orbit $(k_1, l_1)$ with $k_1<0$. But by monotonicity moduli spaces of Maslov $2$ planes are compact, so we again deduce the existence of a $J_1$ holomorphic plane with $k_1<0$. This gives a contradiction as before.

In the second case, both the cylinder and plane have area $2$, which is minimal, and so they cannot degenerate further. We deduce the existence of a building with an unmatched orbit in the original class $(k,l)$. This is enough for a contradiction.

\proofend
\end{proof}

\subsubsection{Foliations by holomorphic spheres and neck stretching.}

Returning to Proposition \ref{int}, we argue by contradiction and assume there exists an embedding $\overline{B(3)} \hookrightarrow B(5)$ with image disjoint from $L(1,2) \cup L(2,2)$.

We can blow-up the image of $B(3)$ and carry out a symplectic reduction on the boundary $\partial B(5)$ to arrive at a copy of $X = \CC P^2 \# \overline{ \CC P^2}$ where the exceptional divisor $E$ has area $3$ and the line at infinity $\Sigma$ has area $5$. By results of McDuff \cite{mcd} all such symplectic manifolds are symplectomorphic. We work with almost complex structures $J$ which are of a fixed standard form near $E$ and $\Sigma$, in particular such that $E$ and $\Sigma$ are complex. It is known that, for any such almost complex structure, $X$ is foliated by $J$-holomorphic spheres in the class $[F] = [\Sigma - E]$. The spheres have area $2$. By positivity of intersection, each sphere in the class intersects each of $\Sigma$ and $E$ exactly once, positively and transversally. Thus a foliation leads to a projection map $p_J : X \to \Sigma$, mapping a point to the intersection of its leaf with $\Sigma$.

Stretching the neck along $L(1,2) \cup L(2,2)$ our foliation converges to a limiting finite energy foliation of $X \setminus (L(1,2) \cup L(2,2))$. Generic leaves are still embedded holomorphic spheres, but now there is a codimension $1$ set of broken leaves. A broken leaf consists of a pair of finite energy planes which are asymptotic to the same geodesics on either $L(1,2)$ or $L(2,2)$, but with the opposite orientation. The planes have area $1$ and compactify to give Maslov $2$ disks. (Here we exclude the possibility of more complicated degenerations because all curves have integral area.) By positivity of intersection, in a broken leaf exactly one of the planes intersects $\Sigma$ and exactly one intersects $E$. If we blow down $E$ then the plane intersecting it transforms to a plane of area $4$ and Maslov class $4$. Together with Lemma \ref{mas}, this imposes strong restrictions on the broken planes, which we describe in the following lemma.

\newpage

\begin{lemma}\label{broken}
\leavevmode
\begin{enumerate}
\item The broken leaves asymptotic to $L(1,2)$ consist of one plane intersecting both $\Sigma$ and $E$ and another plane lying in $X \setminus (\Sigma \cup E)$. The planes are asymptotic to geodesics in the class $(\pm 1,0)$;\\
\item The broken leaves asymptotic to $L(2,2)$ consist of one plane intersecting $\Sigma$ and a second plane intersecting $E$. The planes are asymptotic to geodesics in the class $\pm(1,1)$.\\
\end{enumerate}
\end{lemma}

\begin{proof} For part (1), arguing by contradiction, if there exists a broken plane intersecting $E$ but not $\Sigma$ then it can be blown down to give a holomorphic Maslov $4$ disk of area $4$ in $X \setminus \Sigma \subset \CC^2$. This contradicts Lemma \ref{mas} (4). Therefore one plane in a broken leaf avoids both $E$ and $\Sigma$, and hence represents a Maslov $2$ class in $\CC^2$. By Lemma \ref{mas} (1) the boundary represents the class $(1,0)$.

For part (2), we recall from Lemma \ref{mas} (5) that there are no Maslov $2$ disks in $\CC^2$ with boundary on $L(2,2)$ and area $1$. The first part of the statement follows. Given this, one of the broken planes intersects $E$ but not $\Sigma$, and can be blown down to give a Maslov $4$ disk in $\CC^2$. The second part of the statement then comes from Lemma \ref{mas} (6).
\proofend
\end{proof}

\subsubsection{Notation and the solid tori of broken planes.}

Following \cite{hk2}, we will denote the union of broken planes asymptotic to $L(1,2)$  and disjoint from $E$ and $\Sigma$ by $\frak{r}_0$, and the union of planes asymptotic to $L(1,2)$ and intersecting $E$ and $\Sigma$ by $\frak{r}_{\infty}$. The union broken planes asymptotic to $L(2,2)$ and intersecting $E$ is denoted $\frak{s}_0$ and the union of broken planes asymptotic to $L(2,2)$ and intersecting  $\Sigma$ is denoted $\frak{s}_{\infty}$. We have that each of $\frak{r}_0$, $\frak{r}_{\infty}$, $\frak{s}_0$ and $\frak{s}_{\infty}$ is a solid $3$ dimensional torus $S^1 \times D^2 \subset X$.

Also, we write $\Gamma = p_J(L(1,2)) = p_J(\frak{r}_0 \cup \frak{r}_{\infty})$ and $\Lambda = p_J(L(2,2)) = p_J(\frak{s}_0 \cup \frak{s}_{\infty})$. These are disjoint embedded circles in $\Sigma$, and divide $\Sigma$ into three regions: a disk $A$ with $\partial A = \Gamma$, a cylinder $B$ with $\partial B = \Gamma \cup \Lambda$ and a disk $C$ with $\partial C = \Lambda$. We fix points $p_0 \in A$ and $p_{\infty} \in C$ and set $T_0 = p_j^{-1}(p_0)$ and $T_{\infty} = p_J^{-1}(p_{\infty})$. Then $T_0$ and $T_{\infty}$ are holomorphic spheres in the class $[F]$, that is, they are unbroken leaves of our foliation.

The complement $X \setminus (E \cup \Sigma \cup T_0 \cup T_{\infty})$ is symplectomorphic to a bounded subset of the cotangent bundle $T^* T^2$. Under this identification, $L(2,2)$ becomes a homologically nontrivial Lagrangian torus in $T^* T^2$, but because $\frak{r}_0$ is disjoint from $E$ and $\Sigma$, and from $T_0$ and $T_{\infty}$, the torus $L(1,2) \subset T^* T^2$ is homologically trivial. Our goal is to find a new set of axes in a new copy of $\CC P^2 \# \overline{ \CC P^2}$, such that our two tori are both homologically nontrivial in the complement. We will eventually arrive at a contradiction to Gromov's theorem that exact Lagrangians in a cotangent bundle must intersect the zero section.

\subsubsection{High degree holomorphic spheres.}

Let $d \in \NN$ and $J$ an almost complex structure on $X$ as above. For any generic collection of $2d+2$ points in $X$ we can find a unique $J$-holomorphic sphere in the class $[\Sigma + dF]$ intersecting the $2d+2$ points. We can place $d+1$ points on $L(1,2)$ and $d+1$ points on $L(2,2)$ and take a limit of the holomorphic spheres as we perform a neck stretching of $J$ along $L(1,2) \cup L(2,2)$. The result is a holomorphic building we denote by $\mathbf{F}$. We have the following, as in \cite{karim} Proposition 5.7.

\begin{proposition}\label{buildingF} Let $d$ be very large. Then $\mathbf{F}$ has $2d+3$ (top level) curves in $X \setminus (L(1,2) \cup L(2,2))$. These consist of
\begin{enumerate}
\item a Maslov $2$ plane $D^A$ projecting injectively to $A$;\\
\item a Maslov $2$ plane $D^C$ projecting injectively to $C$;\\
\item a holomorphic cylinder $S^B$ of Fredholm index $2$, projecting injectively to $B$;\\
\item $d$ broken planes asymptotic to $L(1,2)$;\\
\item $d$ broken planes asymptotic to $L(2,2)$.\\
\end{enumerate}
\end{proposition}

Using Lemma \ref{mas} we can determine the areas of the curves $D^A$, $S^B$ and $D^C$. (These are called essential curves in \cite{hk2}, Definition 3.12.),

\begin{lemma}\label{essa} $\mathrm{area}(D^A) =2$, $\mathrm{area}(S^B) =1$, $\mathrm{area}(D^C) =2$.
\end{lemma}

\begin{proof}
We recall from Lemma \ref{broken} that the broken leaves of our foliation asymptotic to $L(1,2)$ are asymptotic to geodesics in the class $(\pm 1, 0)$. Then since
the plane $D^A$ is asymptotic to $L(1,2)$ and projects injectively, its boundary has intersection number $\pm 1$ with the class$(1,0)$ and so lies in a class $(k, \pm 1)$ for some $k \in \ZZ$. First suppose $D^A$ is asymptotic to an orbit in a class $(k,1)$, intersects $\Sigma$ a total of $m$ times and intersects $E$ a total of $n$ times. Then $K=D^A - m \Sigma + nE$ represents a relative homology class in $\CC^2$ with Maslov class $\mu(K) = 2 - 6m + 2n$. By Lemma \ref{mas}(2) it has area $$\mathrm{area}(K) = \mathrm{area}(D^A) - 5m + 3n = 2 - 3m + n$$ and so $\mathrm{area}(D^A) = 2+2m - 2n$ is even. If $D^A$ is asymptotic to an orbit in a class $(k,-1)$ and intersects $\Sigma$ a total of $m$ times and $E$ a total of $n$ times, then with $K$ as before, Lemma \ref{mas}(3) says it has area $$\mathrm{area}(K) = \mathrm{area}(D^A) - 5m + 3n = -3m + n.$$ Thus $\mathrm{area}(D^A) = 2m - 2n$ which again is even.

A similar analysis applies to $D^C$. This plane is asymptotic to $L(2,2)$, where the boundaries of broken planes represent the classes $\pm(1,1)$. The boundary of $D^C$ has intersection number $\pm 1$ with the boundary of broken planes, and so represents a class $(k, k + \epsilon)$, where $\epsilon = \pm 1$. Suppose $D^C$ intersects $\Sigma$ a total of $m$ times and intersects $E$ a total of $n$ times. Then by Lemma \ref{mas}(5) the class $K$ defined as above has area $$\mathrm{area}(K) = \mathrm{area}(D^C) - 5m + 3n = 2 - 6m + 2n$$ and so $\mathrm{area}(D^C) = 2 -m - n$. As holomorphic curves have positive area, and our intersections are positive, we see that either $m=n=0$ and $\mathrm{area}(D^C) = 2$, or else either $(m,n)=(0,1)$ or $(m,n)=(1,0)$. But if $(m,n)=(0,1)$ then $$2= \mu(D^C) = 2k + 2k + 2\epsilon -2$$ and if $(m,n)=(1,0)$ then $$2= \mu(D^C) = 2k + 2k + 2\epsilon + 6.$$ Since $\epsilon = \pm 1$ these equations have no solutions for $k$.

The $2d$ broken planes in the holomorphic building each have area $1$, while homology class $[\Sigma + dF]$ has symplectic area $5 + 2d$. We conclude that
$$5 + 2d =  \mathrm{area}(D^A) +  \mathrm{area}(S^B) + \mathrm{area}(D^C) +2d.$$ We have already shown $\mathrm{area}(D^A) \ge 2$ and $\mathrm{area}(D^C)=2$. Hence $D^A$ must have area $2$ exactly, and the holomorphic cylinder $S^B$ has area $1$.
\proofend
\end{proof}

We can repeat the limiting process above for spheres in the class $[\Sigma + (d-1)F]$ which intersect $d-1$ points on $L(1,2)$ and $d-1$ points on $L(2,2)$. Proposition \ref{buildingF} then gives a holomorphic building $\mathbf{G}$ of the same form, but now with $d-1$ broken planes asymptotic to each of $L(1,2)$ and $L(2,2)$.

Now we deform the buildings $\mathbf{F}$ and $\mathbf{G}$ to produce smooth symplectic spheres with controlled intersection properties. The result of this is described by \cite{karim}, Proposition 5.8.

\newpage

\begin{proposition}\label{deform} There exist smooth spheres $F$ and $G$ in the classes $[\Sigma + dF]$ and $[\Sigma + (d-1)F]$ respectively, with the following properties.
\begin{enumerate}
\item exactly one of $F$ or $G$ intersects the planes in $\frak{r}_0$ and the other intersects the planes in $\frak{r}_{\infty}$;\\
\item exactly one of $F$ or $G$ intersects the planes in $\frak{s}_0$ and the other intersects the planes in $\frak{s}_{\infty}$;\\
\item $F$ and $G$ intersect $d$ times positively and transversally in $p_J^{-1}(A)$ and $d$ times positively and transversally in $p_J^{-1}(C)$;\\
\item $F$ intersects $S^B$ positively in a single point while $G$ and $S^B$ are disjoint.\\
\end{enumerate}
\end{proposition}

The proof of Proposition \ref{deform} is quite lengthy, but can be understood intuitively. The idea is that $F$ is a smoothing of $\mathbf{F}$ into the complement of $L(1,2) \cup L(2,2)$ and $G$ is a similar smoothing of $\mathbf{G}$.

Suppose $\mathbf{F}$ has $a_F$ planes in $\frak{r}_0$ and $b_F$ planes in $\frak{r}_{\infty}$. By Proposition \ref{buildingF} we have $a_F + b_F = d$. Similarly,  $\mathbf{G}$ has $a_G$ planes in $\frak{r}_0$ and $b_G$ planes in $\frak{r}_{\infty}$, with $a_F + b_F = d-1$. Thus $(a_F + b_G) + (b_F + a_G) = 2d-1$ and either $a_F + b_G \ge d$ or $b_F + a_G \ge d$.

In the first case we deform $\mathbf{F}$ (thinking of the building as a subset of $X$, which is smooth away from $L(1,2) \cup L(2,2)$) such that it intersects the planes in $\frak{r}_{\infty}$ and the planar components of the building are pushed into $p_J^{-1}(A)$, and we deform $\mathbf{G}$ such that it intersects the planes in $\frak{r}_0$ and the planar components of the building are also pushed into $p_J^{-1}(A)$. Then the $a_F$ planar components of  $\mathbf{F}$ in $\frak{r}_0$ will intersect the deformation of $\mathbf{G}$, and the $b_G$ planar components of  $\mathbf{G}$ in $\frak{r}_{\infty}$ will intersect the deformation of $\mathbf{F}$. In total our deformed buildings will intersect at least $d$ times in $p_J^{-1}(A)$. In the second case we make an analogous deformation with the deformation of $\mathbf{F}$ intersecting $\frak{r}_0$ and the deformation of $\mathbf{G}$ intersecting $\frak{r}_{\infty}$. Again this will produce at least $d$ intersections in $p_J^{-1}(A)$. The same calculation can be used to push the buildings $\mathbf{F}$ and $\mathbf{G}$ away from $L(2,2)$ giving $d$ intersections in $p_J^{-1}(C)$.

Now, after perhaps perturbing $J$ slightly, there exist holomorphic spheres $F$ and $G$ which coincide with the deformations of $\mathbf{F}$ and $\mathbf{G}$ away from their singular sets. As the relevant homology classes satisfy $[\Sigma + dF] \bullet [\Sigma + (d-1)F] = 2d$ we deduce that in fact there are exactly $d$ intersections in $p_J^{-1}(A)$, exactly $d$ in $p_J^{-1}(C)$, and none in $p_J^{-1}(B)$. Since $F$ is a deformation of $\mathbf{F}$ this last count implies that $G$ must be disjoint from $S^B$.

Finally, a calculation based on $D^A$ and $D^C$ having Maslov class $2$ implies that $F$ may be assumed to intersect $\mathbf{F}$ precisely $d$ times in the closures of $p_J^{-1}(A)$ and $p_J^{-1}(C)$, see \cite{hk2} Lemma 3.29. This includes the intersections of $F$ with broken planes asymptotic to $L(1,2)$ and $L(2,2)$. As $[\Sigma + dF] \bullet [\Sigma + dF] = 2d+1$, we see that $F$ must intersect $S^B$ exactly once.

\subsubsection{Blowing up and down.}

Let $H_1, \dots, H_{2d}$ be the fibers of our $J$ holomorphic foliation through the points $F \cap G$. By Proposition \ref{deform} exactly $d$ of these fibers project to $A$ and $d$ project to $C$. We can blow up $2d$ small balls, say of capacity $\delta$, around the points $F \cap G$ and then blow down the transforms $\hat{H}_i$ of the corresponding fibers. We arrive at a new copy $Y$ of $\CC P^2 \# \overline{ \CC P^2}$ containing the proper transforms $\hat{F}$ and $\hat{G}$ of $F$ and $G$. Now $\hat{F} \bullet \hat{F} = 2d+1 - 2d =1$ and $\hat{G} \bullet \hat{G} = 2d-1 - 2d =-1$. The proper transforms $\hat{T}_0$ and $\hat{T}_{\infty}$ of the fibers $T_0$ and $T_{\infty}$ have self intersection $0$ and lie in the class $[\hat{F} - \hat{G}]$. Appealing again to \cite{mcd}, we can think of $\hat{F}$ as the line at infinity, of area $2d(1-\delta) + 5$ and $\hat{G}$ as the exceptional divisor, of area $2d(1-\delta) +3$, in $Y$. Then the complement $Z:= Y \setminus (\hat{F} \cup \hat{G} \cup \hat{T}_0 \cup \hat{T}_{\infty})$ can be identified with a subset of $T^* T^2$. As $L(1,2) \cup L(2,2) \subset Z$ we can consider $L(1,2)$ and $L(2,2)$ now also as tori in $T^* T^2$.

\subsubsection{Intersection argument.}

It follows from properties (1) and (2) in Proposition \ref{deform} that $L(1,2)$ and $L(2,2)$ become homologically nontrivial in $Z$ and also $T^* T^2$, see \cite{karim}, Lemma 4.11 and Proposition 6.3. We claim they are relatively exact, that is, any $2$ dimensional surface in $T^* T^2$ with boundary on $L(1,2) \cup L(2,2)$ has area $0$.

We can find a plane $P_1$ in either $\frak{r}_0$ or $\frak{r}_{\infty}$ and a second plane $P_2$ in either $\frak{s}_0$ or $\frak{s}_{\infty}$, both of which intersect $\hat{F}$ but are disjoint from $\hat{G} \cup \hat{T}_0 \cup \hat{T}_{\infty}$. Up to isotopy we may assume $P_1$ and $P_2$ coincide on a disk $D$ intersecting $\hat{F}$. Then $(P_1 \setminus D) - (P_2 \setminus D)$ is a cycle in $T^* T^2$ with boundary representing the class $\pm (1,0)$ in $L(1,2)$ and $\pm (1,1)$ in $L(2,2)$. As both $P_1$ and $P_2$ have area $1$ the cycle has area $0$.

Next we consider the cylinder $S^B$. This has area $1$ by Lemma \ref{essa}, and again intersects $\hat{F}$ while avoiding $\hat{G} \cup \hat{T}_0 \cup \hat{T}_{\infty}$. Adding a copy of $-P_1$, we can find another cycle in $T^* T^2$ of area $0$ with boundary on $L(1,2) \cup L(2,2)$, but whose boundary components now represent classes transverse to the foliation class. This implies $L(1,2)$ and $L(2,2)$ are indeed relatively exact.

Our proof concludes as follows. First, by \cite{DGI}, Theorem B, we can apply a Hamiltonian isotopy of $T^* T^2$ mapping $L(1,2)$ to a constant section. Composing with a translation in the fiber we get a global symplectomorphism $\phi$ of $T^* T^2$ with $\phi( L(1,2)) = \mathbb{O}$, the zero section. The image $\phi( L(2,2))$ is now an exact Lagrangian torus. But by \cite{gr} Theorem $2.3.B_4''$ exact Lagrangians must intersect the zero section, and this gives our contradiction.


\end{document}